\documentclass[reqno,a4paper,11pt]{amsart}
\usepackage{amsmath, amsthm, amscd, amsfonts, amssymb, graphicx, color}
\usepackage{amssymb}
\usepackage[english]{babel}
\usepackage[dvips, lmargin=4cm, rmargin=4cm, tmargin=4cm, bmargin=4cm]{geometry}
\usepackage[colorlinks=true,urlcolor=blue,linkcolor=blue, citecolor=red]{hyperref}
\usepackage{hyperref}
\usepackage[T1]{fontenc}
\usepackage{amssymb}
\usepackage{tikz}
\usetikzlibrary{backgrounds}
\usepackage[english]{babel}
\usepackage{enumitem}
\usepackage{dsfont}
\newtheorem{thm}{Theorem}[section]
\newtheorem{defini}{Definition}[section]
\newtheorem{rem}{Remark}[section]
\newtheorem{lem}{Lemma}[section]
\newtheorem{prop}{Proposition}[section]

\newtheorem{ex}{Example} [section]

\numberwithin{equation}{section}

\def \R {\mathbb{R} }

\begin{document}
\title[Nonlocal problems with Neumann and Robin boundary condition]{Nonlocal problems with Neumann and Robin boundary condition in fractional Musielak-Sobolev spaces}
\author[E. Azroul, A. Benkirane,  and  M. Srati]
{E. Azroul$^1$, A. Benkirane$^2$  and  M. Srati$^3$}
\address{E. Azroul, A. Benkirane and  M. Srati\newline
 Sidi Mohamed Ben Abdellah
 University,
 Faculty of Sciences Dhar El Mahraz, Laboratory of Mathematical Analysis and Applications, Fez, Morocco.}
\email{$^1$elhoussine.azroul@gmail.com}
\email{$^2$abd.benkirane@gmail.com}
\email{$^3$mohammed.srati@usmba.ac.ma}
\maketitle
\begin{abstract}
 In this paper, we develop some properties of the $a_{x,y}(.)$-Neumann derivative for the fractional $a_{x,y}(.)$-Laplacian   operator. Therefore we prove the basic proprieties of the correspondent function spaces. In the second part of this paper, by means of Ekeland's variational principal and direct variational approach, we prove the existence of weak solutions for a nonlocal problem with nonhomogeneous Neumann and Robin boundary condition.  
\end{abstract}
\tableofcontents
\section{Introduction}\label{S1}
 In the last years, great attention has been devoted to the study of nonlinear problems involving nonlocal operators  in modular spaces. In particular, in the fractional Orlicz-Sobolev spaces $W^sL_\varPhi(\Omega)$ (see \cite{sr5,3,SRN1,SRN2,SRT,sr_mo,sal1,sal2}) and in the fractional Sobolev spaces  with variable exponents $W^{s,p(x,y)}(\Omega)$ (see \cite{SH,SH2,SRH,SS2020,athman,ku}). The study of variational problems where the modular function  satisfies  nonpolynomial growth conditions instead of having the usual $p$-structure arouses much interest in the development of applications to electrorheological fluids as an important class of non-Newtonian fluids (sometimes referred to as smart fluids). The electro-rheological fluids are characterized by their ability to drastically change the mechanical properties under the influence of an external electromagnetic field. A mathematical model of electro-rheological fluids was proposed by Rajagopal and Ruzicka (we refer the reader to \cite{maria1,maria2,e2} for more details). Another important application is related to image processing \cite{e3} where this kind of diffusion operator is used to underline the borders of the distorted image and to eliminate the noise. From a  mathematical standpoint, it is a hard task to show the existence of classical solutions, i.e., solutions which are continuously differentiable as many times as the order of the differential equations under consideration. However, the concept of weak solution is not enough to give a formulation of all problems and may not provide existence or stability properties.

The Neumann boundary condition, credited to the German mathematician Neumann, is also known as the boundary condition of the second kind. In this type of boundary condition, the value of the gradient of the dependent variable normal to the boundary, $\frac{\partial \phi}{\partial n}$, is prescribed on the boundary. 

In the last years, great attention has been devoted to the study of nonlocal problems with fractional Neumann boundary condition, In this contex, Dipierro, Ros-Oton, and Valdinoci, in \cite{N5}  introduce an extension for the classical Neumann condition $\frac{\partial \phi}{\partial n} = 0$ on $\partial\Omega$ consists in the nonlocal prescription

\begin{equation}\label{n1}
     \begin{aligned}
     \mathcal{N}^s_2u(x)= \int_{\Omega}  \dfrac{u(x)-u(y)}{|x-y|^{N+2s}}dy ,~~\forall x\in \R^N\setminus \Omega.
               \end{aligned}
\end{equation}
Other Neumann problems for the fractional Laplacian (or other nonlocal operators) were introduced in \cite{N1,N2,N3}. All these different Neumann problems for nonlocal operators recover the classical Neumann problem as a limit case, and most of them have clear probabilistic interpretations as well. 
An advantage of this approach (\ref{n1}) is that the problem has a variational structure. 

In \cite{N4}, Mugnai and Proietti Lippi  introduced an extension of (\ref{n1}) as following
\begin{equation}\label{n2}
     \begin{aligned}
     \mathcal{N}^s_pu(x)= \int_{\Omega}  \dfrac{|u(x)-u(y)|^{p-2}(u(x)-u(y))}{|x-y|^{N+ps}}dy ,~~\forall x\in \R^N\setminus \Omega,
               \end{aligned}
\end{equation}
$\mathcal{N}^s_p$ is the nonlocal normal $p$-derivative, or $p$-Neumann boundary condition and describes the natural Neumann
boundary condition in presence of the fractional $p$-Laplacian. It extends the notion of nonlocal normal derivative for the fractional Laplacian, i.e. for $p = 2$. In this situation, $p > 1$, $s \in (0, 1)$.

In fractional modular spaces,  Bahrouni,  Radulesc\u{u}, and  Winkert in \cite{N6} defined the following boundary condition
\begin{equation}\label{n3}
      \begin{aligned}
      \mathcal{N}^s_{p(x,.)}u(x)= \int_{\Omega}  \dfrac{|u(x)-u(y)|^{p(x,y)-2}(u(x)-u(y))}{|x-y|^{N+sp(x,y)} }dy,~~\forall x\in \R^N\setminus \Omega,
                \end{aligned}
\end{equation}

where $p: \R^{2N} \longrightarrow (1, +\infty)$ is a symmetric, continuous function bounded and  $p(.) = p(.,.)$. $\mathcal{N}^s_{p(x,.)}$   is the nonlocal normal $p(.,.)$-derivative [or $p(.,.)$-Neumann boundary condition] and describes the natural Neumann boundary condition in the presence of the fractional $p(.,.)$-Laplacian, (\ref{n2}) extends the notion of the nonlocal normal derivative  for the fractional $p$-Laplacian.

On other extention  of $p$-Neumann boundary condition, has proposed by Bahrouni and Salort in \cite{N7} as following
{\small$$
     \begin{aligned}
     \mathcal{N}^s_{a(.)}u(x)= \int_{\Omega} a\left( \dfrac{|u(x)-u(y)|}{|x-y|^s }\right)\dfrac{u(x)-u(y)}{|x-y|^s} \dfrac{dy}{|x-y|^{N+s}},~~\forall x\in \R^N\setminus \Omega,
               \end{aligned}
                $$ }
where $a = A'$ such that $A$ is a Young function   and $s \in (0, 1)$.
 
In this paper, we introduce the natural Neumann boundary condition in the presence of the fractional $a_{x,y}(.)$-Laplacian in fractional Musielak Sobolev spaces. Therefore  we are concerned with the existence of  weak solutions to the following Neumann-Robin problem
 $$\label{P}
 (\mathcal{P}_a)  \left\{ 
    \begin{array}{clclc}
 (-\Delta)^s_{a_{(x,.)}} u +\widehat{a}_x(|u|)u& = & \lambda f(x,u)   & \text{ in }& \Omega, \\\\
     \mathcal{N}^s_{a(x,.)}u+\beta(x)\widehat{a}_x(|u|)u & = & 0 \hspace*{0.2cm}  & \text{ in } & \R^N\setminus \Omega,
    \end{array}
    \right. 
 $$
  where $\Omega$ is an open bounded subset in $\R^N$, $N\geqslant 1$,   with Lipschitz boundary $\partial \Omega$, $0<s<1$, 
 $f: \Omega\times \R \longrightarrow \R$ is a Carath\'eodory function, $\beta\in L^{\infty}(\R^N\setminus \Omega)$ such that $\beta\geqslant 0$ in $\R^N\setminus \Omega$ and  $(-\Delta)^s_{a_{(x,y)}}$ is the nonlocal integro-differential operator of elliptic type defined as follows
    {\small  $$
               \begin{aligned}
               (-\Delta)^s_{a_{(x,.)}}u(x)=2\lim\limits_{\varepsilon\searrow 0} \int_{\R^N\setminus B_\varepsilon(x)} a_{(x,y)}\left( \dfrac{|u(x)-u(y)|}{|x-y|^s }\right)\dfrac{u(x)-u(y)}{|x-y|^s} \dfrac{dy}{|x-y|^{N+s} },
               \end{aligned}
                $$}  
  for all $x\in \R^N$, where $(x,y,t)\mapsto a_{(x,y)}(t):=a(x,y,t) : \overline{\Omega}\times\overline{\Omega}\times \R\longrightarrow \R$   is symmetric function :
 \begin{equation}\label{n4}
 a(x,y,t)=a(y,x,t) ~~ \forall(x,y,t)\in \overline{\Omega}\times\overline{\Omega}\times \R,\end{equation}
   and the function : $\varphi(.,.,.) : \overline{\Omega}\times\overline{\Omega}\times \R \longrightarrow \R$ defined by  
$$
  \varphi_{x,y}(t):=\varphi(x,y,t)= \left\{ 
          \begin{array}{clclc}
        a(x,y,|t|)t   & \text{ for }& t\neq 0, \\\\
          0  & \text{ for } & t=0,
          \end{array}
          \right. 
$$
is increasing homeomorphism from $\R$ onto itself. Let 
$$\varPhi_{x,y}(t):=\varPhi(x,y,t)=\int_{0}^{t}\varphi_{x,y}(\tau)d\tau~~\text{ for all } (x,y)\in \overline{\Omega}\times\overline{\Omega},~~\text{ and all } t\geqslant 0.$$  
Then, $\varPhi_{x,y}$ is a Musielak function (see \cite{mu}), that is
\begin{itemize}
\item  $\varPhi(x,y,.)$ is a $\varPhi$-function for every $(x,y)\in\overline{\Omega}\times\overline{\Omega}$, i.e.,   is continuous, nondecreasing function with $\varPhi(x,y,0)= 0$, $\varPhi(x,y,t)>0$ for $t>0$ and $\varPhi(x,y,t)\rightarrow \infty$ as $t\rightarrow \infty$.
\item For every $t\geqslant 0$, $\varPhi(.,.,t) : \overline{\Omega}\times\overline{\Omega} \longrightarrow \R$ is a measurable function.
\end{itemize}
Also, we take $ \widehat{a}_x(t):=\widehat{a}(x,t)=a_{(x,x)}(t)  ~~ \forall~ (x,t)\in \overline{\Omega}\times \R$. Then the function $\widehat{\varphi}(.,.) : \overline{\Omega}\times \R \longrightarrow \R$ defined  by : 
  $$
     \widehat{\varphi}_{x}(t):=\widehat{\varphi}(x,t)= \left\{ 
          \begin{array}{clclc}
        \widehat{a}(x,|t|)t   & \text{ for }& t\neq 0, \\\\
          0  & \text{ for } & t=0,
          \end{array}
          \right. 
       $$
is increasing homeomorphism from $\R$ onto itself. If we set 
\begin{equation}\label{phi}
\widehat{\varPhi}_{x}(t):=\widehat{\varPhi}(x,t)=\int_{0}^{t}\widehat{\varphi}_{x}(\tau)d\tau ~~\text{ for all}~~ t\geqslant 0.
\end{equation}  
Then, $\widehat{\varPhi}_{x}$ is also a Musielak function.
 
Furthermore, $\mathcal{N}^s_{a(x,.)}$ is defined by
{\small$$
     \begin{aligned}
     \mathcal{N}^s_{a(x,.)}u(x)= \int_{\Omega} a_{(x,y)}\left( \dfrac{|u(x)-u(y)|}{|x-y|^s }\right)\dfrac{u(x)-u(y)}{|x-y|^s} \dfrac{dy}{|x-y|^{N+s}},~~\forall x\in \R^N\setminus \Omega,
               \end{aligned}
                $$ }
                  denotes $a_{(x,.)}-$Neumann boundary condition and present the natural Neumann boundary condition for $(-\Delta)^s_{a_{(x,.)}}$ in fractional Musielak-Sobolev space.  
                  
                  If we take, $a_{x,y}(t)=t^{p(x,y)-2}$, this work extends the notion of the nonlocal normal derivative for the fractional $p(.,.)$-Laplacian, and if $a_{(x,y)}(t)=a(t)$, i.e. the function $a$ is independent of variables $x,y$ so this work extends  the notion of the nonlocal normal derivative  for the fractional $a(.)$-Laplacian, and therefore this work extends the notion of the nonlocal normal derivative for the fractional Laplacian operator $(\ref{n1})$ and for fractional $p$-Laplacian operator $(\ref{n2})$.
                  
       This paper is organized as follows, In Section \ref{S1}, we  set  the problem  \hyperref[P]{$(\mathcal{P}_{a})$} and the related hypotheses. Moreover,  we are introduced the new  Neumann boundary condition  associated  to fractional $a_{x,y}(.)$-Laplacian operator.  The Section \ref{S2}, is devoted to recall
   some properties of fractional Musielak-Sobolev spaces. In section \ref{S3}, we introduce the corresponding function space for weak solutions of \hyperref[P]{$(\mathcal{P}_{a})$}, and we prove some properties, and state
   the corresponding Green formula for problems such as \hyperref[P]{$(\mathcal{P}_{a})$}. In section \ref{S4},    by means of Ekeland's variational principle and direct variational approach,
       we obtain the existence  of $\lambda^*>\lambda_*>0$ such that for any $\lambda\in(0,\lambda_*)\cup [\lambda^*,\infty)$, problem \hyperref[P]{$(P_a)$} has a nontrivial weak solution.  Finally, in Section \ref{S5}, we present some examples which illustrate our  results.              
        \section{Preliminaries results}\label{S2}                      
To deal with this situation we define the fractional Musielak-Sobolev space to
investigate Problem \hyperref[P]{$(\mathcal{P}_{a})$}. Let us recall the definitions and some elementary
properties of this spaces. We refer the reader to \cite{benkirane,benkirane2} for further reference
and for some of the proofs of the results in this section.

 For the function $\widehat{\varPhi}_x$ given in (\ref{phi}), we introduce the Musielak space as follows
  $$L_{\widehat{\varPhi}_x} (\Omega)=\left\lbrace u : \Omega \longrightarrow \R \text{ mesurable }: \int_\Omega\widehat{\varPhi}_x(\lambda |u(x)|)dx < \infty \text{ for some } \lambda>0 \right\rbrace. $$
The space $L_{\widehat{\varPhi}_x} (\Omega)$ is a Banach space endowed with the Luxemburg norm 
$$||u||_{\widehat{\varPhi}_x}=\inf\left\lbrace \lambda>0 \text{ : }\int_\Omega\widehat{\varPhi}_x\left( \dfrac{|u(x)|}{\lambda}\right) dx\leqslant 1\right\rbrace. $$
 The conjugate function of $\varPhi_{x,y}$ is defined by $\overline{\varPhi}_{x,y}(t)=\int_{0}^{t}\overline{\varphi}_{x,y}(\tau)d\tau~~\text{ for all } (x,y)\in\overline{\Omega}\times\overline{\Omega}~~ \text{ and all } t\geqslant 0$, where $\overline{\varphi}_{x,y} : \R\longrightarrow \R$ is given by $\overline{\varphi}_{x,y}(t):=\overline{\varphi}(x,y,t)=\sup\left\lbrace s \text{ : } \varphi(x,y,s)\leqslant t\right\rbrace.$ Furthermore, we have the following H\"older type inequality
  \begin{equation}
   \left| \int_{\Omega}uvdx\right| \leqslant 2||u||_{\widehat{\varPhi}_x}||v||_{\overline{\widehat{\varPhi}}_x}\hspace*{0.5cm} \text{ for all } u \in L_{\widehat{\varPhi}_x}(\Omega)  \text{ and } v\in L_{\overline{\widehat{\varPhi}}_x}(\Omega).
   \end{equation}
    Throughout this paper, we assume that there exist two positive constants $\varphi^+$ and $\varphi^-$ such that 
\begin{equation}\label{v1}\tag{$\varPhi_1$}
    1<\varphi^-\leqslant\dfrac{t\varphi_{x,y}(t)}{\varPhi_{x,y}(t)}\leqslant \varphi^+<+\infty\text{ for all } (x,y)\in\overline{\Omega}\times\overline{\Omega}~~\text{ and all } t\geqslant 0. \end{equation}
    This relation implies  that
    \begin{equation}\label{A2}
        1<\varphi^-\leqslant \dfrac{t\widehat{\varphi}_{x}(t)}{\widehat{\varPhi}_{x}(t)}\leqslant\varphi^+<+\infty,\text{ for all } x\in\overline{\Omega}~~\text{ and all } t\geqslant 0.\end{equation}
             It follows that  $\varPhi_{x,y}$ and $\widehat{\varPhi}_{x}$ satisfy the global $\Delta_2$-condition (see \cite{ra}), written $\varPhi_{x,y}\in \Delta_2$ and $\widehat{\varPhi}_{x}\in \Delta_2$, that is,
    \begin{equation}\label{r1}
    \varPhi_{x,y}(2t)\leqslant K_1\varPhi_{x,y}(t)~~ \text{ for all } (x,y)\in\overline{\Omega}\times\overline{\Omega},~~\text{ and  all } t\geqslant 0,
    \end{equation} and
    \begin{equation}\label{rr1}
        \widehat{\varPhi}_{x}(2t)\leqslant K_2\widehat{\varPhi}_{x}(t) ~~\text{ for any } x\in\overline{\Omega},~~\text{ and  all } t\geqslant 0,
        \end{equation}
 where $K_1$ and $K_2$ are two positive constants. 
 
 Furthermore, we assume that $\varPhi_{x,y}$ satisfies the following condition
  \begin{equation}\label{f2.}\tag{$\varPhi_2$}
  \text{ the function } [0, \infty) \ni t\mapsto \varPhi_{x,y}(\sqrt{t}) \text{ is convex. }
  \end{equation}
 
    \begin{defini}
     Let $A_x(t)$, $B_x(t): \R^+\times \Omega\longrightarrow \R^+$ be two Musielak functions. 
            $A_x$ is stronger $($resp essentially stronger$)$ than $B_x$,  $A_x\succ B_x$ (resp $A_x\succ\succ B_x$) in symbols, if for almost every $x\in \overline{\Omega}$ 
          $$B(x,t)\leqslant A( x,a t),~~ t\geqslant t_0\geqslant 0, $$
          for some $($resp for each$)$ $a>0$ and $t_0$ (depending on $a$).
    \end{defini}
    
   Now, due to the nonlocality of the operator $(-\Delta)^s_{a_{(x,.)}}$,  we  define the new fractional Musielak-Sobolev space as introduce in \cite{benkirane} as follows 
    \begingroup\makeatletter\def\f@size{9}\check@mathfonts$$ W^s{L_{\varPhi_{x,y}}}(\Omega)=\Bigg\{u\in L_{\widehat{\varPhi}_x}(\Omega) :  \int_{\Omega} \int_{\Omega} \varPhi_{x,y}\left( \dfrac{\lambda| u(x)- u(y)|}{|x-y|^s}\right) \dfrac{dxdy}{|x-y|^N}< \infty \text{ for some } \lambda >0 \Bigg\}.
$$\endgroup
This space can be equipped with the norm
\begin{equation}\label{r2}
||u||_{s,\varPhi_{x,y}}=||u||_{\widehat{\varPhi}_x}+[u]_{s,\varPhi_{x,y}},
\end{equation}
where $[.]_{s,\varPhi_{x,y}}$ is the Gagliardo seminorm defined by 
$$[u]_{s,\varPhi_{x,y}}=\inf \Bigg\{\lambda >0 :  \int_{\Omega} \int_{\Omega} \varPhi_{x,y}\left( \dfrac{|u(x)- u(y)|}{\lambda|x-y|^s}\right) \dfrac{dxdy}{|x-y|^N}\leqslant 1 \Bigg\}.
$$

\begin{thm}$($\cite{benkirane}$)$.
       Let $\Omega$ be an open subset of $\R^N$, and let $s\in (0,1)$. The space $W^sL_{\varPhi_{x,y}}(\Omega)$ is a Banach space with respect to the norm $(\ref{r2})$, and a  separable $($resp. reflexive$)$ space if and only if $\varPhi_{x,y} \in \Delta_2$ $($resp. $\varPhi_{x,y}\in \Delta_2 $ and $\overline{\varPhi}_{x,y}\in \Delta_2$$)$. Furthermore,
              if   $\varPhi_{x,y} \in \Delta_2$ and $\varPhi_{x,y}(\sqrt{t})$ is convex, then  the space $W^sL_{\varPhi_{x,y}}(\Omega)$ is an uniformly convex space.\end{thm}

           \begin{defini}$($\cite{benkirane}$)$.
           We say that $\varPhi_{x,y}$ satisfies the fractional boundedness condition, written $\varPhi_{x,y}\in \mathcal{B}_{f}$, if
         \begin{equation}\tag{$\varPhi_3$}
        \label{v3}         
           \sup\limits_{(x,y)\in \overline{\Omega}\times\overline{\Omega}}\varPhi_{x,y}(1)<\infty.  \end{equation}
           \end{defini}
           \begin{thm}  $($\cite{benkirane}$)$.    \label{TT}
                       Let $\Omega$ be an open subset of $\R^N$,  and  $0<s<1$. Assume that  $\varPhi_{x,y}\in \mathcal{B}_{f}$. 
                       Then,
                       $$C^2_0(\Omega)\subset W^sL_{\varPhi_{x,y}}(\Omega).$$
                  \end{thm}
       
   
                     For any $u \in W^sL_{\varPhi_{x,y}}(\Omega)$, we define the modular function on  $W^sL_{\varPhi_{x,y}}(\Omega)$  as follows  
                    \begin{equation}\label{modN}
        \varPsi(u)=\displaystyle\int_{\Omega} \int_{\Omega} \varPhi_{x,y}\left( \dfrac{ |u(x)- u(y)|}{|x-y|^s}\right) \dfrac{dxdy}{|x-y|^N}+\int_{\Omega}\widehat{\varPhi}_{x}\left( |u(x)|\right) dx. \end{equation}  
                     
   An important role in manipulating  the fractional Musielak-Sobolev spaces is played by the modular function $(\ref{modN})$ . It is worth noticing that the relation between the norm and the modular shows an equivalence between the topology defined by the norm and that defined by the modular.                                     
        \begin{prop}$($\cite{benkirane}$)$.\label{mod}
         Assume that (\ref{v1}) is satisfied. Then, for any $u \in W^sL_{\varPhi_{x,y}}(\Omega)$, the following relations hold true:
           \begin{equation}\label{mod1}
     ||u||_{s,\varPhi_{x,y}}>1\Longrightarrow      ||u||_{s,\varPhi_{x,y}}^{\varphi^-} \leqslant  \varPsi(u)\leqslant  ||u||_{s,\varPhi_{x,y}}^{\varphi^+},
           \end{equation}
           \begin{equation}\label{mod2}
                ||u||_{s,\varPhi_{x,y}}<1\Longrightarrow    ||u||_{s,\varPhi_{x,y}}^{\varphi^+} \leqslant  \varPsi(u)\leqslant  ||u||_{s,\varPhi_{x,y}}^{\varphi^-}. \end{equation}
           \end{prop}  
  We denote by $\widehat{\varPhi}_{x}^{-1}$ the inverse function of $\widehat{\varPhi}_{x}$ which satisfies the following conditions:
       \begin{equation}\label{15}
       \int_{0}^{1} \dfrac{\widehat{\varPhi}_{x}^{-1}(\tau)}{\tau^{\frac{N+s}{N}}}d\tau<\infty~~ \text{ for all } x\in \overline{\Omega},
       \end{equation}
       
       \begin{equation}\label{16n}
       \int_{1}^{\infty} \dfrac{\widehat{\varPhi}_{x}^{-1}(\tau)}{\tau^{\frac{N+s}{N}}}d\tau=\infty ~~\text{ for all }x\in \overline{\Omega}.
       \end{equation}
      Note that, if $\varphi_{x,y}(t)=|t|^{p(x,y)-1}$, then (\ref{15}) holds precisely when $sp(x,y)<N$ for all $(x,y)\in \overline{\Omega}\times \overline{\Omega}$.\\
       If (\ref{16n}) is satisfied, we define the inverse  Musielak conjugate function of $\widehat{\varPhi}_x$ as follows
       \begin{equation}\label{17}
       (\widehat{\varPhi}^*_{x,s})^{-1}(t)=\int_{0}^{t}\dfrac{\widehat{\varPhi}_{x}^{-1}(\tau)}{\tau^{\frac{N+s}{N}}}d\tau.
       \end{equation}
        \begin{thm}\cite{benkirane2}\label{3.4}
      Let $\Omega$  be a bounded open
       subset of  $\R^N$ with $C^{0,1}$-regularity 
         and bounded boundary. If $(\ref{15})$ and  $(\ref{16n})$  hold, then 
      \begin{equation}\label{18}
       W^s{L_{\varPhi_{x,y}}}(\Omega)\hookrightarrow L_ {\widehat{\varPhi}^*_{x,s}}(\Omega).
      \end{equation}
     \end{thm}
      \begin{thm}\cite{benkirane2}\label{th2.}
               Let $\Omega$  be a bounded open
                 subset of  $\R^N$ and  $C^{0,1}$-regularity 
                   with bounded boundary. If $(\ref{15})$ and  $(\ref{16n})$  hold, then the embedding
                \begin{equation}\label{27}
                 W^s{L_{\varPhi_{x,y}}}(\Omega)\hookrightarrow L_{B_x}(\Omega),
                \end{equation}
                is compact for all $B_x\prec\prec \widehat{\varPhi}^*_{x,s}$.
                \end{thm}
   
        Finally, the proof of our existence result is based on the following Ekeland's variational principle theorem and direct variational approach.
   \begin{thm}\label{th1}(\cite{ek})
   Let V be a complete metric space and $F : V \longrightarrow \R\cup \left\lbrace +\infty\right\rbrace$ be a lower semicontinuous functional on $V$, that is bounded below and not identically equal to $+\infty$. Fix $\varepsilon>0$ and a  point $u\in V$ 
     such that
    $$F(u)\leqslant \varepsilon +\inf\limits_{x\in V}F(x).$$ Then for every $\gamma > 0$,
     there exists some point $v\in V$ such that :
     $$F(v)\leqslant F(u),$$
     $$d(u,v)\leqslant \gamma$$
     and for all $w\neq v$
     $$F(w)> F(v)-\dfrac{\varepsilon}{\gamma}d(v,w).$$
   \end{thm}
   
    \begin{thm}\label{th2} (\cite{110})
        Suppose that $Y$ is a reflexive Banach space with norm $||.||$ and let
        $V\subset Y$ be a weakly closed subset of $Y$. Suppose $E : V \longrightarrow \R \cup \left\lbrace +\infty\right\rbrace $ is coercive
        and (sequentially) weakly lower semi-continuous on $V$ with respect to $Y$, that
        is, suppose the following conditions are fulfilled:
     \begin{itemize}
        \item[$\bullet$] $E(u)\rightarrow \infty$ as $||u||\rightarrow \infty$, $u\in V$.
        
        \item[$\bullet$]  For any $u\in V$, any sequence $\left\lbrace u_n\right\rbrace $ in $V$ such that $u_n\rightharpoonup u$ weakly in $X$
        there holds:
       $$E(u)\leqslant \liminf_{n\rightarrow \infty}E(u_n).$$
         \end{itemize} 
        Then $E$ is bounded from below on $V$ and attains its infimum  in $V$.     
         \end{thm}        
\section{Some qualitative properties of  $\mathcal{N}^s_{a(x,.)}$}\label{S3}
The aim of this section is to give the basic properties of the fractional $a_{a(x,)}$-Laplacian with the associated $a_{a(x,)}$-Neumann boundary
condition.
          
Let $u : \R^N \longrightarrow \R$ be a measurable function, we set
$$\|u\|_X=[u]_{s,\varPhi_{x,y},\R^{2N}\setminus (C\Omega)^2}+\|u\|_{\widehat{\varPhi}_x}+\|u\|_{\widehat{\varPhi}_x,\beta,C\Omega}$$ 
    where
{\small$$[u]_{s,\varPhi_{x,y},\R^{2N}\setminus (C\Omega)^2}=\inf \Bigg\{\lambda >0 :  \int_{\R^{2N}\setminus (C\Omega)^2} \varPhi_{x,y}\left( \dfrac{|u(x)- u(y)|}{\lambda|x-y|^s}\right) \dfrac{dxdy}{|x-y|^N}\leqslant 1 \Bigg\}
$$ }   
and
$$\|u\|_{\widehat{\varPhi}_x,\beta,C\Omega}=\inf\left\lbrace \lambda>0 \text{ : }\int_{C\Omega}\beta(x)\widehat{\varPhi}_x\left( \dfrac{|u(x)|}{\lambda}\right) dx\leqslant 1\right\rbrace $$
    with $C\Omega =\R^N\setminus \Omega$. We define 
    $$X=\left\lbrace u : \R^N\longrightarrow \R~~\text{ measurable } : \|u\|_X<\infty\right\rbrace.$$ 
    \begin{rem}
 It is easy to see that $\|.\|_X$ is a norm on $X$. We only show that if $\|u\|_X=0$, then $u=0$ a.e. in $\R^N$. Indeed, form $\|u\|_X=0$, we get $\|u\|_{\widehat{\varPhi}_x}=0$, which implies that 
 \begin{equation}\label{N1}
 u=0 ~~\text{a.e. in } \Omega
 \end{equation} 
 and 
 \begin{equation}\label{N2}
 \int_{\R^{2N}\setminus (C\Omega)^2} \varPhi_{x,y}\left( \dfrac{|u(x)- u(y)|}{|x-y|^s}\right) \dfrac{dxdy}{|x-y|^N}=0.
 \end{equation} 
 By $(\ref{N2})$, we deduce that $u(x)=u(y)$ in $\R^{2N}\setminus (C\Omega)^2$, that is $u=c\in \R$ in $\R^N$, and by $(\ref{N1})$ we have $u=0$ a.e. in $\R^N$.
    
      \end{rem} 
    \begin{prop}\label{rem}
     Note that the norm $\|.\|_X$ is equivalent on $X$ to
     $$\|u\|:=\inf\left\lbrace \lambda>0~~:~~\rho_s\left( \dfrac{u}{\lambda}\right) \leqslant 1\right\rbrace $$
     where, the modular function $\rho_s~~ : X\longrightarrow \R$ is defined by
     $$
     \begin{aligned}
     \rho_s(u)=&\int_{\R^{2N}\setminus (C\Omega)^2} \varPhi_{x,y}\left( \dfrac{|u(x)- u(y)|}{|x-y|^s}\right) \dfrac{dxdy}{|x-y|^N}\\
     &+\int_{\Omega}\widehat{\varPhi}_x\left(|u(x)|\right) dx+\int_{C\Omega}\beta(x)\widehat{\varPhi}_x\left( |u(x)|\right) dx.
      \end{aligned}
      $$     
     \end{prop}
      Proof is similar to \cite[Proposition 2.1]{benkirane}.
    \begin{prop}\label{Nmod}
             Assume that (\ref{v1}) is satisfied. Then, for any $u \in X$, the following relations hold true:
               \begin{equation}\label{Nmod1}
         ||u||>1\Longrightarrow      ||u||^{\varphi^-} \leqslant  \rho_s(u)\leqslant  ||u||^{\varphi^+},
               \end{equation}
               \begin{equation}\label{Nmod2}
                    ||u||<1\Longrightarrow    ||u||^{\varphi^+} \leqslant  \rho_s(u)\leqslant  ||u||^{\varphi^-}. \end{equation}
               \end{prop}  
 Proof is similar to \cite[Proposition 2.2]{benkirane}.
\begin{prop}
$\left(X, \|.\|_X\right)$ is a reflexive Banach space.
\end{prop} 
\begin{proof} 
Now, we prove that $X$ is complete. For this, let $\left\lbrace u_n\right\rbrace $ be a Cauchy sequence in $X$. In particular  $\left\lbrace u_n\right\rbrace $ is a Cauchy sequence in $L_{\widehat{\varPhi}_x(\Omega)}$ and so, there exists $u\in L_{\widehat{\varPhi}_x(\Omega)}$ such that 
$$u_n\longrightarrow u~~\text{in}~~L_{\widehat{\varPhi}_x(\Omega)}~~\text{and a.e. in } \Omega.$$
Then, we can find $Z_1\subset \R^N$ such that 
\begin{equation}\label{N3}
|Z_1|=0~~\text{ and } u_n(x)\longrightarrow u(x)~~\text{ for every } x\in \Omega\setminus Z_1.
\end{equation} 
For any $u : \R^N\longrightarrow \R$, and for any $(x,y)\in \R^{2N}$, we set
$$E_u(x,y)=\dfrac{(u(x)-u(y))}{|x-y|^s}\mathcal{X}_{\R^{2N}\setminus (C\Omega)^2}(x,y).$$
Using the fact that $\left\lbrace u_n\right\rbrace $ is a Cauchy sequence in $L_{\varPhi_{x,y}}\left( \R^{2N},d\mu\right)$, where $\mu$   is a  measure on  $\Omega\times\Omega$ which is given by
           $d\mu :=|x-y|^{-N}dxdy.$ So, there exists a subsequence $\left\lbrace E_{u_n}\right\rbrace$ converges to $E_u$ in $L_{\varPhi_{x,y}}\left( \R^{2N},d\mu\right)$ and a.e. in $\R^{2N}$. Then, we can find $Z_2\subset \R^{2N}$ such that
\begin{equation}\label{N4}
|Z_2|=0~~\text{ and } E_{u_n}(x,y)\longrightarrow E_u(x,y)~~\text{ for every } (x,y)\in \R^{2N}\setminus Z_2.
\end{equation} 
For any $x\in \Omega$, we set
$$S_x:=\left\lbrace y\in \R^N~~:~~(x,y)\in \R^{2N}\setminus Z_2\right\rbrace $$
$$W:=\left\lbrace (x,y)\in \R^{2N},~~x\in \Omega~~\text{and}~~y\in \R^N\setminus S_x\right\rbrace $$
$$V:=\left\lbrace x\in \Omega~~:~~|\R^N\setminus S_x|=0\right\rbrace.$$
Let $(x,y)\in W$, we have $y\in \R^N\setminus S_x$. Then $(x,y)\notin \R^{2N}\setminus Z_2$, i.e. $(x,y)\in Z_2$. So
$$W\subset Z_2,$$
therefore, by $(\ref{N4})$
$$|W|=0,$$
then, by the Fubini's Theorem we have
$$0=|W|=\int_{\Omega}| \R^N\setminus S_x|dx,$$
which implies that $| \R^N\setminus S_x|=0$ a.e $x\in \Omega$. It follows that $|\Omega\setminus V|=0$. This end with $(\ref{N3})$, implies that
$$|\Omega\setminus (V\setminus Z_1)|=|(\Omega\setminus V)\cup Z_1|\leqslant |\Omega\setminus V|+|Z_1|=0.$$
In particular $V\setminus Z_1\neq \varnothing,$ then we can fix $x_0\in V\setminus Z_1$, and by $(\ref{N3})$, it follows
$$\lim\limits_{n\rightarrow \infty} u_n(x_0)=u(x_0).$$
In addition, since $x_0\in V,$ we obtain $|\R^N\setminus S_{x_0}|=0$. Then, for almost all $y\in \R,$ this yields $(x_0,y)\in \R^{2N}\setminus Z_2$, and hence, by
$(\ref{N4})$ 
$$\lim\limits_{n\rightarrow \infty} E_{u_n}(x_0,y)=E_u(x_0,y).$$
Since $\Omega\times C\Omega \subset \R^{2N}\setminus (C\Omega)^2$, we have 
$$E_{u_n}(x_0,y)=\dfrac{(u_n(x_0)-u_n(y))}{|x_0-y|^s}\mathcal{X}_{\R^{2N}\setminus (C\Omega)^2}(x_0,y)$$
for almost all $y\in C\Omega.$ However, this implies
$$\lim\limits_{n\rightarrow \infty} u_n(y)=\lim\limits_{n\rightarrow \infty}\left(  u_n(x_0) -|x_0-y|^sE_{u_n}(x_0,y)\right)= u(x_0) -|x_0-y|^sE_{u}(x_0,y)$$
for almost all $y\in C\Omega.$
Combining this end with $(\ref{N3})$, we see that $u_n$ is converges to some $u$ a.e. in $\R^N$. Since $u_n$ is a Cauchy sequence in $X$, so for any $\varepsilon>0$, there exists $N_\varepsilon>0$ such that for any $k>N_\varepsilon$, we have by applying Fatou's Lemma
$$
\begin{aligned}
\varepsilon\geqslant & \liminf\limits_{k\rightarrow \infty}\|u_n-u_k\|_X\\
&\geqslant c\liminf\limits_{k\rightarrow \infty}\|u_n-u_k\|\\
&\geqslant c \liminf\limits_{k\rightarrow \infty}\left( \rho_{s}(u_n-u_k)\right) ^{\frac{1}{\varphi^\pm}}\\
&\geqslant c \left( \rho_{s}(u_n-u)\right) ^{\frac{1}{\varphi^\pm}}\\
&\geqslant c \|u_n-u\|^{\frac{\varphi^\pm}{\varphi^\pm}}\\
&\geqslant c \|u_n-u\|_X^{\frac{\varphi^\pm}{\varphi^\pm}},
\end{aligned}
$$
where $c$ is a positive constant given by Proposition $\ref{rem}$. This implies that $u_n$ converge to $u$ in $X$, and so $X$ is complete space.
Now, we show that $X$ is a reflexive space. For this, we consider the following space
$$Y=L_{\widehat{\varPhi}_x}(\Omega)\times L_{\widehat{\varPhi}_x}(C\Omega)\times L_{\widehat{\varPhi}_{x,y}}\left( \R^{2N}\setminus (C\Omega)^2,d\mu\right) $$
endowed with the norm
$$\|u\|_Y=[u]_{s,\varPhi_{x,y},\R^{2N}\setminus (C\Omega)^2}+\|u\|_{\widehat{\varPhi}_x}+\|u\|_{\widehat{\varPhi}_x,\beta,C\Omega}.$$
We note that $(Y, \|.\|_Y)$ is a reflexive Banach space, we consider the map 
$T : X\longrightarrow Y$ defined as :
$$T(u)=\left(u,u,D^su\right).$$
By construction, we have that
$$\|T(u)\|_Y=\|u\|_X.$$
Hence, $T$ is an isometric from $X$ to the reflexive space $Y$. This show that $X$ is reflexive.
\end{proof}
\begin{prop}\label{N10}

      Let $\Omega$  be a bounded open
       subset of  $\R^N$ with $C^{0,1}$-regularity 
         and bounded boundary. If $(\ref{15})$ and  $(\ref{16n})$  hold, then 
      \begin{equation}\label{N18}
       X\hookrightarrow L_ {\widehat{\varPhi}^*_{x,s}}(\Omega).
      \end{equation}
              In particular, the embedding
                \begin{equation}\label{N27}
                 X\hookrightarrow L_{B_x}(\Omega),
                \end{equation}
                is compact for all $B_x\prec\prec \widehat{\varPhi}^*_{x,s}$.
         
\end{prop}  
 \begin{proof}
 Since $\Omega\times\Omega\subset \R^{2N}\setminus (C\Omega)^2$. Then
 $$||u||_{s,\varPhi_{x,y}}\leqslant \|u\|_X~~\text{for all }~~u\in X.$$
 Therefore, by Theorems \ref{3.4} and \ref{th2.}, we get our desired result.
 \end{proof}

         Now, by integration by part formula, we have the following result.
         \begin{prop}
         Let $u\in X$, then
         $$\int_\Omega  (-\Delta)^s_{a_{(x,.)}}u(x) dx=-\int_{\R^N\setminus \Omega}\mathcal{N}^s_{a(x,.)}u(x)dx.$$
         \end{prop} 
         \begin{proof}  
       Since the role of $x$ and $y$ are symmetric and $a_{x,y}$ is a symmetric function, we obtain
       $$
       \begin{aligned}
     \int_{\Omega} \int_{\Omega} & a_{(x,y)}\left( \dfrac{|u(x)-u(y)|}{|x-y|^s }\right)\dfrac{u(x)-u(y)}{|x-y|^s} \dfrac{dxdy}{|x-y|^{N+s}}\\
     &= - \int_{\Omega} \int_{\Omega} a_{(x,y)}\left( \dfrac{|u(x)-u(y)|}{|x-y|^s }\right)\dfrac{u(y)-u(x)}{|x-y|^s} \dfrac{dxdy}{|x-y|^{N+s}}\\
     &=- \int_{\Omega} \int_{\Omega} a_{(y,x)}\left( \dfrac{|u(y)-u(x)|}{|x-y|^s }\right)\dfrac{u(x)-u(y)}{|x-y|^s} \dfrac{dydx}{|x-y|^{N+s}}\\
     &=- \int_{\Omega} \int_{\Omega} a_{(x,y)}\left( \dfrac{|u(x)-u(y)|}{|x-y|^s }\right)\dfrac{u(x)-u(y)}{|x-y|^s} \dfrac{dxdy}{|x-y|^{N+s}}.\\
       \end{aligned}  
       $$
       This implies that 
       $$2\int_{\Omega} \int_{\Omega}  a_{(x,y)}\left( \dfrac{|u(x)-u(y)|}{|x-y|^s }\right)\dfrac{u(x)-u(y)}{|x-y|^s} \dfrac{dxdy}{|x-y|^{N+s}}=0$$
       that is,
      $$ \int_{\Omega} \int_{\Omega}  a_{(x,y)}\left( \dfrac{|u(x)-u(y)|}{|x-y|^s }\right)\dfrac{u(x)-u(y)}{|x-y|^s} \dfrac{dxdy}{|x-y|^{N+s}}=0.$$
   Hence, we have that
  $$
   \begin{aligned}
   \int_\Omega  (-\Delta)^s_{a_{(x,.)}}u(x) dx &=\int_{\Omega} \int_{\R^N}  a_{(x,y)}\left( \dfrac{|u(x)-u(y)|}{|x-y|^s }\right)\dfrac{u(x)-u(y)}{|x-y|^s} \dfrac{dydx}{|x-y|^{N+s}}\\
   &= \int_{\Omega} \int_{\R^N\setminus \Omega}  a_{(x,y)}\left( \dfrac{|u(x)-u(y)|}{|x-y|^s }\right)\dfrac{u(x)-u(y)}{|x-y|^s} \dfrac{dydx}{|x-y|^{N+s}}\\
   & ~~+\int_{\Omega} \int_{\Omega}  a_{(x,y)}\left( \dfrac{|u(x)-u(y)|}{|x-y|^s }\right)\dfrac{u(x)-u(y)}{|x-y|^s} \dfrac{dydx}{|x-y|^{N+s}}\\
   &= \int_{\R^N\setminus \Omega} \left(  \int_{\Omega}  a_{(x,y)}\left( \dfrac{|u(x)-u(y)|}{|x-y|^s }\right)\dfrac{u(x)-u(y)}{|x-y|^s} \dfrac{dx}{|x-y|^{N+s}}\right) dy\\
   &= -\int_{\R^N\setminus \Omega}\mathcal{N}^s_{a(x,.)}u(y)dy.
   \end{aligned} 
   $$
  \end{proof} 
    \begin{prop}
  For all $u\in X$, we have
  $$ 
   \begin{aligned}
   \dfrac{1}{2} & \int_{\R^{2N}\setminus(C\Omega)^2}  a_{(x,y)}\left( \dfrac{|u(x)-u(y)|}{|x-y|^s }\right)\dfrac{u(x)-u(y)}{|x-y|^s}\dfrac{v(x)-v(y)}{|x-y|^s} \dfrac{dxdy}{|x-y|^{N}}\\
   &=\int_\Omega v (-\Delta)^s_{a_{(x,.)}}u dx+\int_{C \Omega} v \mathcal{N}^s_{a(x,.)}udx.\\
   \end{aligned}
   $$                     
                    \end{prop}        
    \begin{proof}
    By symmetric, and since $\R^{2N}\setminus(C\Omega)^2=(\Omega\times \R^N)\cup (C \Omega\times\Omega)$. Then, we have
\begin{equation}\label{N7}
  \begin{aligned}
  \dfrac{1}{2} & \int_{\R^{2N}\setminus(C\Omega)^2}  a_{(x,y)}\left( \dfrac{|u(x)-u(y)|}{|x-y|^s }\right)\dfrac{u(x)-u(y)}{|x-y|^s}\dfrac{v(x)-v(y)}{|x-y|^s} \dfrac{dxdy}{|x-y|^{N}}\\
  &=\dfrac{1}{2}  \int_{\R^{2N}\setminus(C\Omega)^2} v(x) a_{(x,y)}\left( \dfrac{|u(x)-u(y)|}{|x-y|^s }\right)\dfrac{u(x)-u(y)}{|x-y|^s} \dfrac{dxdy}{|x-y|^{N+s}}\\
  &~~- \dfrac{1}{2}  \int_{\R^{2N}\setminus(C\Omega)^2} v(y) a_{(x,y)}\left( \dfrac{|u(x)-u(y)|}{|x-y|^s }\right)\dfrac{u(x)-u(y)}{|x-y|^s} \dfrac{dxdy}{|x-y|^{N+s}}\\
  &=\dfrac{1}{2}  \int_{\R^{2N}\setminus(C\Omega)^2} v(x) a_{(x,y)}\left( \dfrac{|u(x)-u(y)|}{|x-y|^s }\right)\dfrac{u(x)-u(y)}{|x-y|^s} \dfrac{dxdy}{|x-y|^{N+s}}\\
    &~~- \dfrac{1}{2}  \int_{\R^{2N}\setminus(C\Omega)^2} v(y) a_{(y,x)}\left( \dfrac{|u(x)-u(y)|}{|x-y|^s }\right)\dfrac{u(y)-u(x)}{|x-y|^s} \dfrac{dxdy}{|x-y|^{N+s}}\\
    &= \int_{\R^{2N}\setminus(C\Omega)^2} v(x) a_{(x,y)}\left( \dfrac{|u(x)-u(y)|}{|x-y|^s }\right)\dfrac{u(x)-u(y)}{|x-y|^s} \dfrac{dxdy}{|x-y|^{N+s}}\\
    &=\int_{\Omega}v(x)\int_{\R^{N}} a_{(x,y)}\left( \dfrac{|u(x)-u(y)|}{|x-y|^s }\right)\dfrac{u(x)-u(y)}{|x-y|^s} \dfrac{dxdy}{|x-y|^{N+s}}\\
    &~~ +\int_{C\Omega} v(x) \int_{\Omega} a_{(x,y)}\left( \dfrac{|u(x)-u(y)|}{|x-y|^s }\right)\dfrac{u(x)-u(y)}{|x-y|^s} \dfrac{dxdy}{|x-y|^{N+s}}\\
    &= \int_\Omega v (-\Delta)^s_{a_{(x,.)}}u dx+\int_{C \Omega} v \mathcal{N}^s_{a(x,.)}udx.\\
    \end{aligned}
\end{equation}
                       \end{proof}           
Based on the integration by part formula, we are now in position to state the natural definition of a weak solution of \hyperref[P]{$(\mathcal{P}_{a})$}. First, to simplify the notation, for arbitrary function $u, v\in X$, we set
{\small$$
\begin{aligned}
\mathcal{A}_{s}(u,v)   = & \dfrac{1}{2}  \int_{\R^{2N}\setminus(C\Omega)^2}  a_{(x,y)}\left( \dfrac{|u(x)-u(y)|}{|x-y|^s }\right)\dfrac{u(x)-u(y)}{|x-y|^s}\dfrac{v(x)-v(y)}{|x-y|^s} \dfrac{dxdy}{|x-y|^{N}}\\
&+\int_{\Omega}\widehat{a}_{x}(|u|)u vdx+\int_{C \Omega}\beta(x) \widehat{a}_{x}(|u|)u vdx.
\end{aligned}
$$}       
We  say that $u\in X$ is a weak solution of \hyperref[P]{$(\mathcal{P}_{a})$} is
\begin{equation}\label{N6}
\mathcal{A}_{s}(u,v) =\lambda\int_\Omega f(x,u)vdx
\end{equation}
for all $v\in X$. 
\begin{rem}
Let us first state the definition of a weak solution to our problem $(\ref{N6})$. Note
that here we are using that $a_{x,y}$ is symmetric.  Therefore, In \cite{benkirane,benkirane2}, the authors must set the condition (\ref{n4}), to be the definition of weak solution has a meaning.
\end{rem}

As a consequence of this definition (\ref{N6}), we have the following result.
\begin{prop}
Let $u\in X$ be a weak solution of \hyperref[P]{$(\mathcal{P}_{a})$}. Then
$$\mathcal{N}^s_{a(x,.)}u+\beta(x) \widehat{a}_{x}(|u|)u=0~~\text{ a.e in   } \R^N\setminus \Omega.$$
\end{prop}
\begin{proof}
First, we take $v\in X$ such that $v=0$ in $\Omega$ as a test function in $(\ref{N6})$, and similar calculus to $(\ref{N7})$.  We have 
$$
\begin{aligned}
0=&\mathcal{A}_{s}(u,v)\\
&=  \dfrac{1}{2}  \int_{\R^{2N}\setminus(C\Omega)^2}  a_{(x,y)}\left( \dfrac{|u(x)-u(y)|}{|x-y|^s }\right)\dfrac{u(x)-u(y)}{|x-y|^s}\dfrac{v(x)-v(y)}{|x-y|^s} \dfrac{dxdy}{|x-y|^{N}}\\
&~~+\int_{C \Omega}\beta(x) \widehat{a}_{x}(|u|)u vdx\\
&=\int_{\Omega}  \int_{\R^N\setminus\Omega} a_{(x,y)}\left( \dfrac{|u(x)-u(y)|}{|x-y|^s }\right)\dfrac{u(x)-u(y)}{|x-y|^s}v(x) \dfrac{dxdy}{|x-y|^{N+s}}\\
&~~+\int_{C \Omega}\beta(x) \widehat{a}_{x}(|u|)u vdx\\
&=\int_{\R^N\setminus\Omega}v(x)\int_{\Omega}   a_{(x,y)}\left( \dfrac{|u(x)-u(y)|}{|x-y|^s }\right)\dfrac{u(x)-u(y)}{|x-y|^s} \dfrac{dxdy}{|x-y|^{N+s}}\\
&~~+\int_{C \Omega}\beta(x) \widehat{a}_{x}(|u|)u vdx\\
&=\int_{\R^N\setminus\Omega}v(x)\mathcal{N}^s_{a(x,.)}u(x)dx+\int_{C \Omega}\beta(x) \widehat{a}_{x}(|u|)u vdx\\
&=\int_{\R^N\setminus\Omega}\left( \mathcal{N}^s_{a(x,.)}u(x)dx+\beta(x) \widehat{a}_{x}(|u|)u\right)  v(x) dx.\\
\end{aligned} 
$$
This implies that 
$$\int_{\R^N\setminus\Omega}\left( \mathcal{N}^s_{a(x,.)}u(x)dx+\beta(x) \widehat{a}_{x}(|u|)u\right)  v(x) dx=0$$
for any $v\in X$, and $v=0$ in $\Omega$. In particular is true for every $v\in C^\infty_c(\R^N\setminus \Omega)$, and so
$$\mathcal{N}^s_{a(x,.)}u+\beta(x) \widehat{a}_{x}(|u|)u=0~~\text{ a.e in   } \R^N\setminus \Omega.$$
\end{proof}
    \section{Existence results and proofs}\label{S4}
  The aim of this section is  to prove the existence of a weak solution of \hyperref[P]{$(\mathcal{P}_{a})$}. 
 In what follows, we will work with the modular norm $\|.\|$ and we denote by  $\left( X^*, ||.||_*\right)$         the dual space of $\left( X, ||.||\right)$.\\

    Next, we suppose that  $f : \Omega \times \R \rightarrow  \R$ is a Carath\'eodory function such that       
  \begin{equation}\label{f1}\tag{$f_1$}  |f(x,t)|\leqslant c_1|t|^{q(x)-1},  \end{equation}    
                \begin{equation}\label{f2}\tag{$f_2$} c_2|t|^{q(x)}\leqslant F(x,t):=\int_{0}^{t}f(x,\tau)d\tau,  \end{equation}
for all $x\in \Omega$ and all $t\in \R^N$,  where $c_1$ and $c_2$ are two positive constants, and $q\in C(\overline{\Omega})$ with $1<q^+\leqslant \varphi^-$.
\begin{rem}\label{rem1}
Since $q^+<\varphi^-$  it is easy to see that  $\widehat{\varPhi}_x$ dominates $t\mapsto|t|^{q(x)}$ near infinity. Then by Proposition $\ref{N10}$ the space $X$  is compactly embedded in  $L^{q(x)}(\Omega)$.
 \end{rem} 
\begin{ex}
We point out certain examples of function $f$  which satisfies the hypotheses \hyperref[f1]{$(f_1)$} and \hyperref[f2]{$(f_2)$}.
\begin{itemize}
\item $f(x,t)=q(x)|t|^{q(x)-2}t$, and $F(x,t)=|t|^q(x)$, where $q\in C(\overline{\Omega})$ satisfies $2\leqslant q(x)<p_s^*(x)$ for all $x\in \overline{\Omega}$.

\item $f(x,t)=q(x)|t|^{q(x)-2}t+(q(x)-2)\log(1+t^2)|t|^{q(x)-4}t+\dfrac{t}{1+t^2}|t|^{q(x)-2}$, and $F(x,t)=|t|^{q(x)}+\log(1+t^2)|t|^{q(x)-2}$, where $q\in C(\overline{\Omega})$ satisfies $4\leqslant q<p_s^*(x)$ for all $x\in \overline{\Omega}$.

\item $f(x,t)=q(x)|t|^{q(x)-2}t+(q(x)-1)\sin(\sin t)\times|t|^{q(x)-3}t\cos(\sin t)\cos t |t|^{q(x)-1}$, and $F(x,t)=|t|^q(x)+\sin(\sin t)|t|^{q(x)-1},$ where $q\in C(\overline{\Omega})$ satisfies $3\leqslant q(x)<p_s^*(x)$ for all $x\in \overline{\Omega}$.
\end{itemize}
\end{ex}

   For simplicity, we set
   \begin{equation}
   D^su:=\dfrac{u(x)-u(y)}{|x-y|^s}. \label{h} 
    \end{equation}       
       
   Now, we are ready to state our  existence result. 
      \begin{thm}\label{2.1.}
       Assume $f$ satisfy  \hyperref[f1]{$(f_1)$} and \hyperref[f2]{$(f_2)$}. Then there exist $\lambda_*$ and $\lambda^*$, such that for any $\lambda\in(0,\lambda_*)\cup [\lambda^*,\infty)$, problem \hyperref[P]{$(\mathcal{P}_{a})$} has a nontrivial weak solutions.    \end{thm}
           
    For each $\lambda>0$, we define the energy  functional $J_\lambda :  X\longrightarrow \R$ by
    {\small\begin{equation}\label{14.}
    \begin{aligned}
    J_\lambda(u)=&\displaystyle\dfrac{1}{2}  \int_{\R^{2N}\setminus(C\Omega)^2} \varPhi_{x,y}\left( \dfrac{ |u(x)- u(y)|}{|x-y|^s}\right) \dfrac{dxdy}{|x-y|^N}+\int_{\Omega}\widehat{\varPhi}_x\left( |u(x)|\right) dx\\
    &+\int_{C \Omega}\beta(x) \widehat{\varPhi}_x\left( |u(x)|\right) dx-\lambda\int_{\Omega}F(x,u)dx.
        \end{aligned}
    \end{equation}}
         \begin{rem}
    We note that the functional $J_\lambda :  X\longrightarrow \R$ in $(\ref{14.})$ is well
    defined. Indeed, if $u\in X$, then, we have  $u \in  L^{q(x)}(\Omega)$. Hence, by the condition \hyperref[f1]{$(f_1)$},
    $$ |F(x,u)|\leqslant\int_{0}^{u}|f(x,t)|dt=c_1|u|^{q(x)}$$
    and thus, 
    $$\int_{\Omega}|F(x,u)|dx<\infty.$$
         \end{rem}
   We first  establish some basis properties of $J_\lambda$.
   \begin{prop}\label{prop1}
    Assume condition \hyperref[f1]{$(f_1)$} is satisfied. Then, for each $\lambda>0$, $J_\lambda\in C^1\left( X, \R \right)$ with the derivative given by 
  $$
  \begin{aligned}
  \left\langle J'_\lambda(u),v\right\rangle =& \dfrac{1}{2}  \int_{\R^{2N}\setminus(C\Omega)^2} a_{x,y}(|D^su|)D^su D^svd\mu+\int_{\Omega}\widehat{a}_{x}(|u|)u vdx\\
  &+\int_{C \Omega}\beta(x) \widehat{a}_{x}(|u|)u vdx-\lambda\int_{\Omega}f(x,u)vdx
    \end{aligned}
  $$
   for all $u,v \in X$.\end{prop} 
   Proof of this Proposition is similar to \cite[Proposition 3.1]{benkirane}.\\
   
    Now, define the functionals $I_i : X\longrightarrow \R$ $i=1,2$ by 
   $$
   \begin{aligned}
      I_1(u)=&\displaystyle\dfrac{1}{2}  \int_{\R^{2N}\setminus(C\Omega)^2} \varPhi_{x,y}\left( \dfrac{ |u(x)- u(y)|}{|x-y|^s}\right) \dfrac{dxdy}{|x-y|^N}+\int_{\Omega}\widehat{\varPhi}_x\left( |u(x)|\right) dx\\
       &+\int_{C \Omega}\beta(x) \widehat{\varPhi}_x\left( |u(x)|\right) dx,
           \end{aligned}$$
   and
   $$I_2(u)= \int_{\Omega}F(x,u)dx.$$
   \begin{prop}\label{lem3}
  The functional $J_\lambda$ is weakly lower semi continuous.
   \end{prop}  
    \begin{proof} First, note that $I_1$
       is lower semi-continuous in the weak topology of $X$. Indeed, 
        since $\varPhi_{x,y}$ is a convex function so $I_1$ is also convex. Then, let $\left\lbrace u_n\right\rbrace \subset X$ with $u_n\rightharpoonup u$ weakly in $ X$, then by convexity of $I_1$ we have 
     $$I_1(u_n)-I_1(u)\geqslant \left\langle I_1'(u),u_n-u\right\rangle,$$ 
          	and hence, we obtain $$I_1(u)\leqslant \liminf I_1(u_n),$$ that is, the map $I_1$
    is  weakly lower semi continuous.          
  On the other hand, since $I_2\in  C^1\left( X, \|.\|\right),$ we have
          $$\lim\limits_{n\rightarrow \infty}\int_{\Omega}F(x,u_n)dx=\int_{\Omega}F(x,u)dx.$$
         Thus, we find 
          $$J_\lambda(u)\leqslant \liminf J_\lambda(u_n).$$
          Therefore, $J_\lambda$ is weakly lower semi continuous and Proposition $\ref{lem3}$ is verified. \end{proof}
      \begin{lem}\label{4.4.}
        Assume that the sequence $\left\lbrace u_n\right\rbrace $ converges weakly to $u$ in $X$ and 
        \begin{equation}\label{35.}
        \limsup_{n\rightarrow \infty}\left\langle I_1'(u_n),u_n-u\right\rangle \leqslant 0.
        \end{equation}
        Then the sequence $\left\lbrace u_n\right\rbrace$  is convergence strongly to $u$ in $X$.
        \end{lem}
                         \begin{proof} Since $u_n$ converges weakly to $u$ in $X$, then $\left\lbrace ||u_n||\right\rbrace $ is a bounded sequence of real numbers. Then by Proposition $\ref{Nmod}$, we deduce that $\left\lbrace I_1(u_n)\right\rbrace$ is bounded. So for a subsequence, we deduce that, 
              $$I_1(u_n)\longrightarrow c.$$
              Or since $I_1$ is weak lower semi continuous, we get 
              $$I_1(u)\leqslant \liminf_{n\rightarrow \infty}I_1(u_n)=c.$$
     On the other hand, by the convexity of $I_1$, we have 
     $$I_1(u)\geqslant I_1(u_n)+\left\langle I_1'(u_n),u_n-u\right\rangle .$$   
     Next, by the hypothesis  $(\ref{35.})$, we conclude that $$I_1(u)=c.$$
     Since $\left\lbrace \dfrac{u_n+u}{2} \right\rbrace $ converges weakly to $u$ in $X$, so since $I_1$ is  sequentially weakly lower semicontinuous :
     \begin{equation}\label{32.}
     c=I_1(u)\leqslant \liminf_{n\rightarrow \infty}I_1\left( \dfrac{u_n+u}{2}\right). \end{equation}
      We assume by  contradiction that $\left\lbrace u_n\right\rbrace$ does not converge to $u$ in $X$. Hence,  there exist a subsequence of $\left\lbrace u_n\right\rbrace $, still denoted by $\left\lbrace u_n\right\rbrace $ and there exits $\varepsilon_0>0$ such that 
        $$\Bigg|\Bigg|\dfrac{u_n-u}{2}\Bigg|\Bigg|\geqslant \dfrac{\varepsilon_0}{2},$$
        by Proposition $\ref{Nmod}$, we have
        $$I_1\left( \dfrac{u_n-u}{2}\right) \geqslant\max\left\lbrace \varepsilon_0^{\varphi^-},\varepsilon_0^{\varphi^+}\right\rbrace.$$
        On the other hand, by the conditions (\ref{v1}) and (\ref{f2.}), we can apply \cite[Lemma 2.1]{Lam}  in order to obtain         \begin{equation}\label{33.}
        \dfrac{1}{2}I_1(u_n)+\dfrac{1}{2}I_1(u)-I_1\left( \dfrac{u_n+u}{2}\right) \geqslant I_1\left( \dfrac{u_n-u}{2}\right) \geqslant \max\left\lbrace \varepsilon_0^{\varphi^-}, \varepsilon_0^{\varphi^+}\right\rbrace .
        \end{equation}
        It follows from $(\ref{33.})$ that
        \begin{equation}\label{34.}
        I_1(u)-\max\left\lbrace \varepsilon_0^{\varphi^-}, \varepsilon_0^{\varphi^+}\right\rbrace \geqslant \limsup_{n\rightarrow \infty}I_1\left( \dfrac{u_n+u}{2}\right),
        \end{equation}
        from $(\ref{32.})$ and $(\ref{34.})$ we obtain a contradiction. This shows that $\left\lbrace u_n\right\rbrace$ converges strongly to $u$ in $X$.
           \end{proof}
   \begin{lem}\label{lem5}
   Assume the hypotheses of Theorem $\ref{2.1.}$  are fulfilled. Then there exist $\rho, \alpha>0$ and $\lambda_*>0$ such that for any $\lambda\in (0,\lambda_*),~~J_\lambda(u)\geqslant \alpha>0$ for any $u\in X$ with $||u||=\rho$.
   \end{lem}   
                \begin{proof}  Since $X$ is continuously embedded in $L^{q(x)}(\Omega)$. Then there exists a positive constant $c>0$ such that 
     \begin{equation}\label{28}
     ||u||_{q(x)}\leqslant c||u|| ~~\forall u\in X.
     \end{equation}   
    We fix $\rho\in (0,1)$ such that $\rho<\dfrac{1}{c}$. Then relation $(\ref{28})$ implies that for any $u\in X$ with $||u||=\rho$ : 
    $$\begin{aligned}
    J_\lambda(u)&\geqslant ||u||^{\varphi^+}-\lambda c_2 c^{q^\pm}||u||^{q^\pm}\\
    &=\rho^{q^\pm}\left( \rho^{\varphi^+-{q^\pm}}-\lambda c^{q^\pm} c_2\right).
    \end{aligned} 
     $$
   By the above inequality, we remark if we define 
   \begin{equation}\label{29}
   \lambda_*=\dfrac{\rho^{\varphi^+-{q^\pm}}}{2c_2 c^{q^\pm}}.
   \end{equation}    
     Then for any $u\in X$ with $||u||=\rho$, there exists $\alpha=\dfrac{\rho^{\varphi^+}}{2}>0$ such that 
     $$J_\lambda(u)\geqslant \alpha>0,~~\forall \lambda\in (0,\lambda_*).$$
     The proof of Lemma $\ref{lem5}$ is complete.           \end{proof}
    \begin{lem}\label{lem6}
    Assume the hypotheses of Theorem $\ref{2.1.}$  are fulfilled. Then there exists $\theta\in X$ such that $\theta>0$ and $J_\lambda(t\theta)<0$ for $t>0$ small enough.
    \end{lem} 
      \begin{proof} Let $\Omega_0\subset \subset \Omega$, for $x_0\in \Omega_0$, $0 < R < 1$ satisfy $B_{2R}(x_0)\subset \Omega_0$, where $B_{2R}(x_0)$
               is the ball of radius $2R$ with center at the point $x_0$ in $\R^N$. Let $\theta\in C_0^\infty(B_{2R}(x_0))$ satisfies $0\leqslant \theta \leqslant 1$ and
               $\theta \equiv 1$  in $B_{2R}(x_0)$. Theorem $\ref{TT}$ implies that $||\theta||<\infty.$ Then for $0 < t < 1$, by  $\hyperref[f2]{(f_2)}$, we have    
               $$
                  \begin{aligned}
                         J_\lambda(t\theta)=&
                          \displaystyle\dfrac{1}{2} \int_{\R^{2N}\setminus(C\Omega)^2} \varPhi_{x,y}\left( \dfrac{|t\theta(x)-t\theta(y)|}{|x-y|^s }\right) \dfrac{dxdy}{|x-y|^N} +\int_{\Omega}\varPhi_x(|t \theta|)dx\\
                          &+\int_{C \Omega}\beta(x) \widehat{\varPhi}_{x}(|t\theta|)dx -\lambda\int_{\Omega}F(x,t\theta)dx\\
                         &\leqslant  ||t\theta||^{\varphi^-}-\lambda c_2\int_{\Omega_0} |t\theta|^{q(x)}dx\\
                         &\leqslant t^{\varphi^-}||\theta||^{\varphi^-}-\lambda c_2t^{q\pm}\int_{\Omega_0} |\theta|^{q(x)}dx.
                        \end{aligned}
                        $$
               Since $\varphi^- > q^+$ and $\displaystyle\int_{\Omega_0} |\theta|^{q(x)}dx>0$ we have $J_\lambda(t_0\theta)<0$ for
               $t_0\in(0,t)$ sufficiently small.
                       \end{proof}    
   \begin{lem}\label{lem7}
   Assume the hypotheses of Theorem $\ref{2.1.}$  are fulfilled. Then for any $\lambda>0$ the functional $J_\lambda$ is coercive.
   \end{lem} 
   \begin{proof} For each $u\in X$ with $||u||>1$ and $\lambda>0$, relations $(\ref{mod1})$,  $(\ref{28})$ and the condition \hyperref[f1]{$(f_1)$} imply
      $$
                   \begin{aligned}
                          J_\lambda(u)=&\displaystyle\dfrac{1}{2}  \int_{\R^{2N}\setminus(C\Omega)^2} \varPhi_{x,y}\left( \dfrac{ |u(x)- u(y)|}{|x-y|^s}\right) \dfrac{dxdy}{|x-y|^N}+\int_{\Omega}\widehat{\varPhi}_x\left( |u(x)|\right) dx\\
                                 &+\int_{C \Omega}\beta(x) \widehat{\varPhi}_x\left( |u(x)|\right) dx\\
                          &\geqslant  ||u||^{\varphi^-}-\lambda c_1\int_{\Omega} |u|^{q(x)}dx\\
                          &\geqslant ||u||^{\varphi^-}-\lambda c_1c||u||^{q^\pm}.
                         \end{aligned}
                         $$
    Since $\varphi^->q^+$ the above inequality   implies that $J_\lambda(u)\longrightarrow \infty$ as $||u||\rightarrow \infty$, that is, $J_\lambda$ is coercive.  \end{proof}
         \begin{proof}[\noindent \textbf{Proof of Theorem $\ref{2.1.}$}] Let $\lambda_*>0$ be defined as in $(\ref{29})$ and $\lambda\in (0,\lambda_*)$. By Lemma $\ref{lem5}$ it follows that on the boundary oh the ball centered in the origin and of radius $\rho$ in $X$, denoted by $B_\rho(0)$, we have 
         $$\inf\limits_{\partial B_\rho(0)}J_\lambda>0.$$
      On the other hand, by Lemma $\ref{lem6}$, there exists $\theta \in X$ such that $J_\lambda(t\theta)<0$ for all $t>0$ small enough. Moreover for any $u\in B_\rho(0)$, we have 
      $$
                      \begin{aligned}
                             J_\lambda(u)\geqslant ||u||^{\varphi^+}-\lambda c_1c||u||^q.
                            \end{aligned}
                            $$
      It follows that
      $$-\infty<c:=\inf\limits_{\overline{B_\rho(0)}} J_\lambda<0.$$   
      We let now $0<\varepsilon <\inf\limits_{\partial  B_\rho(0)}  J_\lambda -  \inf\limits_{B_\rho(0)} J_\lambda.$    Applying Theorem $\ref{th1}$ to the functional 
      $J_\lambda : \overline{B_\rho(0)}\longrightarrow \R$, we find $u_\varepsilon \in \overline{B_\rho(0)}$ such that 
       $$
            \left\{ 
                 \begin{array}{clclc}
               J_\lambda(u_\varepsilon)&<\inf\limits_{\overline{B_\rho(0)}} J_\lambda+\varepsilon,& \\\\
                 J_\lambda(u_\varepsilon)&< J_\lambda(u)+\varepsilon ||u-u_\varepsilon||,& \text{  } u\neq u_\varepsilon.
                 \end{array}
                 \right. 
              $$
       Since  $J_\lambda(u_\varepsilon)\leqslant  \inf\limits_{\overline{B_\rho(0)}} J_\lambda+\varepsilon\leqslant \inf\limits_{B_\rho(0)} J_\lambda+\varepsilon \leqslant \inf\limits_{\partial  B_\rho(0)}  J_\lambda$, we deduce $u_\varepsilon  \in B_\rho(0)$. 
       
       Now, we define $\Lambda_\lambda :  \overline{B_\rho(0)}\longrightarrow \R$ by 
       $$\Lambda_\lambda(u)=J_\lambda(u)+\varepsilon||u-u_\varepsilon||.$$
       It's clear that $u_\varepsilon$ is a minimum point of $\Lambda_\lambda$ and then
       $$\dfrac{\Lambda_\lambda(u_\varepsilon+t v)-\Lambda_\lambda(u_\varepsilon)}{t}\geqslant 0$$ 
       for small $t>0$, and any $v\in B_\rho(0).$ The above relation yields 
           $$\dfrac{J_\lambda(u_\varepsilon+t v)-J_\lambda(u_\varepsilon)}{t}+\varepsilon||v||\geqslant 0.$$
           Letting $t\rightarrow 0$ it follows that $\left\langle J'_{\lambda}(u_\varepsilon),v\right\rangle +\varepsilon ||v||>0$ and we infer that $||J'_{\lambda}(u_\varepsilon)||_*\leqslant \varepsilon$.
            We deduce that there exists a sequence $\left\lbrace v_n\right\rbrace \subset B_\rho(0)$ such that 
            \begin{equation}\label{10}
            J_\lambda(v_n) \longrightarrow c \text{ and } J'_\lambda(v_n)\longrightarrow 0.
            \end{equation}
         It is clear that $\left\lbrace v_n\right\rbrace $ is bounded in $X$. Thus, there exists $v\in X$, such that up to a subsequence   $\left\lbrace v_n\right\rbrace $ converges weakly to $v$ in $X$. Since  $X$ is a compactly embedded in $L^{q(x)}(\Omega)$. The above information combined with condition \hyperref[f1]{$(f_1)$} and   H\"older's inequality implies
       \begin{equation}\label{11}
   \begin{aligned}
         \left| \int_{\Omega} f(x,v_n)(v_n-v)dx\right| &\leqslant c_1\int_{\Omega} \left| v_n\right| ^{q(x)-1}\left| v_n-v\right| dx\\   
         &\leqslant c_1\left| \left| |v_n|^{q(x)-1}\right| \right|_{\frac{q(x)}{q(x)-1}}\left| \left| v_n-v\right| \right|_{q(x)} \longrightarrow 0.
         \end{aligned} 
            \end{equation}
    On the other hand, by $(\ref{10})$ we have 
    \begin{equation}\label{12}
    \lim\limits_{n\rightarrow \infty}\left\langle J'_\lambda(v_n) , v_n-v\right\rangle =0.
    \end{equation}  
    Relations $(\ref{11})$ and $(\ref{12})$ imply 
    $$
     \lim\limits_{n\rightarrow \infty}\left\langle I'_1(v_n) , v_n-v\right\rangle =0.$$
     Thus, by Lemma $\ref{4.4.}$ we find that  $\left\lbrace v_n\right\rbrace $ converges strongly to $v$ in $X$, so by $(\ref{10})$: 
     $$ J_\lambda(v)=c<0 \text{ and }  J'_\lambda(v)=0.$$
   We conclude that $v$ is a nontrivial weak solution for problem   \hyperref[P]{$(\mathcal{P}_{a})$} for any $\lambda\in(0,\lambda_*)$.
   
   Next, by Lemma $\ref{lem7}$ and Proposition $\ref{lem3}$ we infer that $J_\lambda$ is coercive and weakly lower semi continuous in $X$ for all $\lambda>0$. Then Theorem $\ref{th2}$ implies that there exists $u_\lambda \in X$ a global minimized of $J_\lambda$ and thus a weak solution of problem \hyperref[P]{$(\mathcal{P}_{a})$}.
   
   Now, we show that $u_\lambda$ is non trivial. Indeed, letting $t_0>1$ be fixed real and
   $$
               \left\{ 
                    \begin{array}{clclc}
                  u_0(x)&=t_0~~\text{ in }  \Omega & \\\\
                    u_0(x)&=0~~\text{ in }  \R^N\setminus \Omega ,&
                    \end{array}
                    \right. 
                 $$
    we have $ u_0\in X$ and 
   $$
   \begin{aligned}
   J_\lambda(u_0)&=I_1(u_0)-\lambda\int_\Omega F(x,u_0)dx\\
   &=\int_\Omega\widehat{\varPhi}_x(t_0)dx-\lambda\int_\Omega F(x,t_0)dx\\
   &\leqslant \int_\Omega \widehat{\varPhi}_x(t_0)dx-\lambda c_2\int_\Omega |t_0|^{q(x)} dx\\
   &=L-\lambda c_2|t_0|^{q^-}|\Omega|,
   \end{aligned}
   $$
   where $L$ is a positive constant. Thus, for $\lambda^*>0$ large enough, $J_\lambda(u_0)<0$ for any $\lambda\in [\lambda^*,\infty)$. It follows that $J_\lambda(u_\lambda)<0$ for any $\lambda\in [\lambda^*,\infty)$ and thus $u_\lambda$ is a nontrivial weak solution of problem  \hyperref[P]{$(\mathcal{P}_{a})$} for any $\lambda\in [\lambda^*,\infty)$. Therefore, problem  \hyperref[P]{$(\mathcal{P}_{a})$} has a nontrivial weak solution for all $\lambda\in(0,\lambda_*)\cup [\lambda^*,\infty)$.\end{proof}
       \section{Examples}  \label{S5} 
       
       In this section we point certain examples of functions $\varphi_{x,y}$ and $\varPhi_{x,y}$  which illustrate the results of this paper.
       \begin{ex}
       	As a first example, we can take
       $$\varphi_{x,y}(t)=p(x,y)|t|^{p(x,y)-2}t ~~\text{and} ~~\varPhi_{x,y}=|t|^{p(x,y)},  ~~\text{ for all}~~ t\geqslant 0,$$  
where $p\in C(\overline{\Omega}\times\overline{\Omega})$ satisfies $2\leqslant p(x,y) <N$  for all $(x,y)\in  \overline{\Omega}\times\overline{\Omega}$.
\end{ex} In this case the problem \hyperref[P]{$(\mathcal{P}_a)$} reduces to the following fractional $p(x,.)$-Laplacian problem  
$$
 (\mathcal{P}_1) \label{P1} \hspace*{0.5cm} \left\{ 
   \begin{array}{clclc}

(-\Delta_{p(x,.)})^{s} u +|u|^{\overline{p}(x)}u& = &  f(x,u)   & \text{ in }& \Omega \\\\
  \mathcal{N}^s_{p(x,.)}u(x)+\beta(x)|u|^{\overline{p}(x)}u & = & 0 \hspace*{0.2cm} \hspace*{0.2cm} & \text{ in } & \R^N\setminus \Omega,
   \end{array}
   \right. 
$$ 
where
$\bar{p}(x)=p(x,x)$ for all $x\in \overline{\Omega}.$
Here, the operator $(-\Delta_{p(x,.)})^{s}$ is the  fractional $p(x,.)$-Laplacian operator defined as follows 
$$(-\Delta_{p(x,.)})^{s}u(x)=p.v.\int_{\Omega}\frac{|u(x)-u(y)|^{p(x,y)-2}(u(x)-u(y))}{|x-y|^{N+sp(x,y)}}~dy~~~~~~ \text{ for all  }  x \in \Omega,$$
 and $\mathcal{N}^s_{p(x,.)}$ is the $p(.,.)$-Neumann boundary condition defined  by
 $$
      \begin{aligned}
      \mathcal{N}^s_{p(x,.)}u(x)= \int_{\Omega}  \dfrac{|u(x)-u(y)|^{p(x,y)}(u(x)-u(y))}{|x-y|^{N+sp(x,y)} }dy,~~\forall x\in \R^N\setminus \Omega.
                \end{aligned}
                 $$
 It easy to see that $\varPhi_{x,y}$ is a Musielak function and   satisfy conditions \hyperref[v1]{$(\varPhi_1)$}-\hyperref[v3]{$(\varPhi_3)$}. In this case we can take $\varphi^-=p^-$ and $\varphi^+=p^+$. 
 Then, we can extract the following result
 \begin{rem}\label{c1}
 	Assume that $f$ satisfies \hyperref[f1]{$(f_1)$} and \hyperref[f2]{$(f_2)$}. If $p^->q^+$. Then, problem \hyperref[P1]{$(\mathcal{P}_1)$}  has a nontrivial weak solution.
 \end{rem}

\begin{ex}
	As a second example, we can take

       $$\varphi_{x,y}(t)=\varphi_{1}(x,y,t) =p(x,y)\dfrac{|t|^{p(x,y)-2}t}{\log (1+|t|)}  ~~\text{ for all}~~ t\geqslant 0,$$ 
       and thus, 
    $$\varPhi_{x,y}(t)=p(x,y)\dfrac{|t|^{p(x,y)}}{\log (1+|t|)}+\int_{0}^{|t|}\dfrac{\tau^{p(x,y)}}{(1+\tau)(\log(1+\tau))^2}d\tau,$$     
    with $p\in C(\overline{\Omega}\times\overline{\Omega})$ satisfies $2\leqslant p(x,y) <N$ for all $(x,y)\in  \overline{\Omega}\times\overline{\Omega}$.
\end{ex}
    Then, in this case problem  \hyperref[P]{$(\mathcal{P}_a)$} becomes 
    $$\label{P2}
     (\mathcal{P}_{2}) \hspace*{0.5cm} \left\{ 
       \begin{array}{clclc}
    
    (-\Delta_{\varphi_{1}})^{s} u +\dfrac{\overline{p}(x)|u|^{\overline{p}(x)-2}u}{\log(1+|u|)}& = &  \lambda f(x,u)   & \text{ in }& \Omega \\\\
       \mathcal{N}^s_{\varphi_1}u(x)+\beta(x)\dfrac{\overline{p}(x)|u|^{\overline{p}(x)-2}u}{\log(1+|u|)} & = & 0 \hspace*{0.2cm} \hspace*{0.2cm} & \text{ in } & \R^N\setminus \Omega,
       \end{array}
       \right. 
    $$ 
   with 
    $$(-\Delta_{\varphi_{1}})^{s}u(x)=p.v.\int_{\Omega}\dfrac{p(x,y)|D^su|^{p(x,y)-2}D^su}{\log(1+|D^su|)|x-y|^{N+s}}~dy~~~~~~ \text{ for all  }  x \in \Omega,$$  
  and 
    $$ \mathcal{N}^s_{\varphi_1}u(x)=\int_{\Omega}\dfrac{p(x,y)|D^su|^{p(x,y)-2}D^su}{\log(1+|D^su|)|x-y|^{N+s}}~dy~~~~~~ \text{ for all  }   x \in \R^N\setminus \Omega.$$  
      It easy to see that $\varPhi_{x,y}$ is a Musielak function and satisfy condition \hyperref[v3]{$(\varPhi_3)$}.
    Moreover, for each $(x,y)\in  \overline{\Omega}\times\overline{\Omega}$ fixed, by Example 3 on p 243 in \cite{cl}, we have 
    $$p(x,y)-1\leqslant \dfrac{t\varphi_{x,y}(t)}{\varPhi_{x,y}(t)}\leqslant p(x,y) ~~\forall (x,y) \in \overline{\Omega}\times\overline{\Omega},~~\forall t\geqslant 0.$$
    Thus, \hyperref[v1]{$(\varPhi_1)$} holds true with $\varphi^-=p^--1$ and $\varphi^+=p^+$.\\
    Finally, we point out that trivial computations imply that 
    $$\dfrac{d^2(\varPhi_{x,y}(\sqrt{t}))}{dt^2}\geqslant 0$$
    for all $(x,y)\in \overline{\Omega}\times \overline{\Omega}$ and $t\geqslant 0$. Thus, relation \hyperref[f2.]{$(\varPhi_2)$} hold true.

    Hence, we derive an existence result for  problem \hyperref[P2]{$(\mathcal{P}_{2})$} which is given by the following Remark.
   \begin{rem}\label{c2}
   	Assume that $f$ satisfies \hyperref[f1]{$(f_1)$} and \hyperref[f2]{$(f_2)$}. If $p^--1>q^+$. Then, problem \hyperref[P2]{$(\mathcal{P}_2)$}  has a nontrivial weak solution. 
   \end{rem}
   \begin{ex}
   	As a third example, we can take  
     $$\varphi_{x,y}(t)=\varphi_2(x,y,t)=p(x,y)\log(1+\alpha+|t|)|t|^{p(x,y)-2}t ~~\text{for all}~~ t\geqslant 0$$
     and  so,
      $$\varPhi_{x,y}(t)=\log(1+|t|)|t|^{p(x,y)}-\int_{0}^{|t|}\dfrac{\tau^{p(x,y)}}{1+\tau}d\tau,$$
      where  $p\in C(\overline{\Omega}\times\overline{\Omega})$ satisfies $2\leqslant p(x,y) <N$ for all $(x,y)\in  \overline{\Omega}\times\overline{\Omega}$.\end{ex}
      
       Then we consider the following fractional $p(x,.)$-problem  
          \begingroup\makeatletter\def\f@size{10}\check@mathfonts $$
          \label{P3} (\mathcal{P}_2)~~  \left\{ 
             \begin{array}{clclc}
          
          (-\Delta_{\varphi_2})^{s} u +\overline{p}(x)\log(1+\alpha+|u|)|u|^{\overline{p}(x)-2}u& = &  f(x,u)   & \text{ in }& \Omega \\\\
             \mathcal{N}^s_{\varphi_2}u(x)+\beta(x)\overline{p}(x)\log(1+\alpha+|u|)|u|^{\overline{p}(x)-2}u & = & 0 ~~  & \text{ in } & \R^N\setminus \Omega,
             \end{array}
             \right. 
          $$ \endgroup
          where 
          $$(-\Delta_{\varphi_2})^{s}u(x)=p.v.\int_{\Omega}\dfrac{p(x,y)\log(1+\alpha+|D^su|).|D^su|^{p(x,y)-2}D^su}{|x-y|^{N+s}}~dy~$$
           for all  $x \in \Omega,$ and 
           
           $$\mathcal{N}^s_{\varphi_2}u(x)=\int_{\Omega}\dfrac{p(x,y)\log(1+\alpha+|D^su|).|D^su|^{p(x,y)-2}D^su}{|x-y|^{N+s}}~dy~$$
      for all  $x \in \R^N\setminus \Omega.$
           
       It easy to see that $\varPhi_{x,y}$ is a Musielak function and satisfy condition \hyperref[v3]{$(\varPhi_3)$}.
       Next, we remark that   for each $(x,y)\in  \overline{\Omega}\times\overline{\Omega}$ fixed,   we have 
       $$p(x,y)\leqslant \dfrac{t\varphi_{x,y}(t)}{\varPhi_{x,y}(t)} ~~\text{for all}~~  t\geqslant 0.$$
 By the above information and  taking $\varphi^-=p^-$, we have 
 $$1<p^-\leqslant \dfrac{t.\varphi_{x,y}(t)}{\varPhi_{x,y}(t)}  \text{  }~~\text{for all}~~(x,y)\in  \overline{\Omega}\times\overline{\Omega} ~~ \text{   }\text{ and all } t\geqslant 0.$$   
  On the other hand, some simple computations imply 
  $$\lim\limits_{t \rightarrow \infty }\dfrac{t.\varphi_{x,y}(t)}{\varPhi_{x,y}(t)} =p(x,y)\text{  }\text{ for all }(x,y)\in \overline{\Omega}  \times\overline{\Omega},$$  
  and 
  $$\lim\limits_{t \rightarrow 0 }\dfrac{t.\varphi_{x,y}(t)}{\varPhi_{x,y}(t)} =p(x,y)+1 \text{  }\text{ for all }(x,y)\in \overline{\Omega}  \times\overline{\Omega},$$
  Thus, we remark that $\dfrac{t.\varphi_{x,y}(t)}{\varPhi_{x,y}(t)}$ is continuous on $\overline{\Omega}  \times\overline{\Omega}\times [0,\infty)$. Moreover, $$1<p^-\leqslant \lim\limits_{t \rightarrow 0 }\dfrac{t.\varphi_{x,y}(t)}{\varPhi_{x,y}(t)}\leqslant p^++1<\infty,$$ and $$1<p^-\leqslant \lim\limits_{t \rightarrow \infty }\dfrac{t.\varphi_{x,y}(t)}{\varPhi_{x,y}(t)}\leqslant p^++1<\infty.$$ It follows that 
  $$\varphi^+<\infty.$$
  We conclude that relation \hyperref[v1]{$(\varPhi_1)$} is satisfied.  Finally, we point out that trivial computations imply that 
      $$\dfrac{d^2(\varPhi_{x,y}(\sqrt{t}))}{dt^2}\geqslant 0$$
      for all $(x,y)\in \overline{\Omega}\times \overline{\Omega}$ and $t\geqslant 0$. Thus, relation \hyperref[f2.]{$(\varPhi_2)$} hold true.

 \begin{rem}\label{c3}
    	Assume that $f$ satisfies \hyperref[f1]{$(f_1)$} and \hyperref[f2]{$(f_2)$}. If $p^->q^+$. Then, problem \hyperref[P3]{$(\mathcal{P}_3)$}  has a nontrivial weak solution.
    \end{rem}


\end{document}